\documentclass[reqno,a4paper,12pt]{amsart}
\pdfoutput=1

\usepackage{mystyle}
\usepackage{etoolbox}

\usepackage[utf8]{inputenc}
\usepackage[T1]{fontenc}
\usepackage{lmodern}
\usepackage[british]{babel}
\usepackage{microtype}
\usepackage[headings]{fullpage}
\usepackage{parskip}

\usepackage[pdfencoding=auto,pdfusetitle]{hyperref}
\hypersetup{colorlinks=true, linkcolor=blue!30!black, citecolor=green!30!black, urlcolor=red!45!black}

\usepackage{microtype}

\usepackage{amsfonts,amssymb,bbm,stmaryrd}
\usepackage[scr=boondoxo,scrscaled=1.05]{mathalfa}
\usepackage{amsmath}

\usepackage{mathtools}
\usepackage[noabbrev,capitalise]{cleveref}
\usepackage{centernot,float,subcaption,url,array,booktabs,colortbl}

\usepackage[shortlabels]{enumitem}
\setlist[enumerate,1]{label={\roman*)}}

\usepackage{tikz}
\usetikzlibrary{matrix,positioning}
\usetikzlibrary{calc}
\usetikzlibrary{arrows}
\tikzset{> =stealth}

\newcommand{\addQEDstyle}[2]{\AtBeginEnvironment{#1}{\pushQED{\qed}\renewcommand{\qedsymbol}{#2}}\AtEndEnvironment{#1}{\popQED}}

\theoremstyle{plain}
\newtheorem{theorem}{Theorem}[section]

\newtheorem{proposition}[theorem]{Proposition}
\newtheorem{corollary}[theorem]{Corollary}

\theoremstyle{definition}
\newtheorem{definition}[theorem]{Definition}

\theoremstyle{remark}
\newtheorem{remark}[theorem]{Remark}

\newtheorem{example}[theorem]{Example}
\addQEDstyle{example}{$\triangle$}

\renewcommand{\epsilon}{\varepsilon}
\renewcommand{\phi}{\varphi}

\mathchardef\mhyphen="2D %

\newcommand{\N}{\mathbb{N}}
\newcommand{\Z}{\mathbb{Z}}
\newcommand{\Q}{\mathbb{Q}}
\newcommand{\R}{\mathbb{R}}

\newcommand{\T}{\mathbb{T}}

\newcommand{\id}{\mathrm{id}}

\newcommand{\Set}{\mathbf{Set}}
\newcommand{\Pos}{\mathbf{Pos}}
\newcommand{\Sup}{\mathbf{Sup}}

\newcommand{\Frm}{\mathbf{Frm}}
\newcommand{\Loc}{\mathbf{Loc}}
\newcommand{\Dcpo}{\mathbf{DCPO}}
\newcommand{\PreFrm}{\mathbf{PreFrm}}

\newcommand*{\dirsup}[1][\hspace{0.75ex}]{\operatorname*{\bigvee{}^{\!\uparrow}_{\hspace{-0.2ex} #1}}}

\renewcommand{\O}{\mathcal{O}}

\newcommand{\llround}{\ensuremath{(\!(}}
\newcommand{\rrround}{\ensuremath{)\!)}}

\title{Presenting quotient locales}

\author[G. Manuell]{Graham Manuell}
\address{CMUC, Department of Mathematics, University of Coimbra, Coimbra, Portugal}
\email{graham@manuell.me}
\thanks{The author acknowledges financial support from the Centre for Mathematics of the University of Coimbra (UIDB/00324/2020, funded by the Portuguese Government through FCT/MCTES)}

\date{July 2023}


\begin{document}

\begin{abstract}
 It is often useful to be able to deal with locales in terms of presentations of their underlying frames, or equivalently, the geometric theories which they classify.
 Given a presentation for a locale, presentations for its sublocales can be obtained by simply appending additional relations, but the case of quotient locales is more subtle.
 We provide simple procedures for obtaining presentations of open quotients, proper quotients or general triquotients from presentations of the parent locale.
 The results are proved with the help of the suplattice, preframe and dcpo coverage theorems and applied to obtain presentations of the circle from ones for $\R$ and $[0,1]$.
\end{abstract}

\maketitle
\thispagestyle{empty}

\setcounter{section}{-1}
\section{Introduction}

An advantage of the pointfree approach to topology is the ability to present frames by generators and relations.
This can also be interpreted as describing a locale by giving axioms for a geometric theory which it classifies.

Since sublocales correspond to quotient frames, presentations for them the be found by simply adding additional relations to the a presentation for the parent frame.
However, there is no such simple relationship between presentations of a frame and presentations of its subframes, and so quotient locales can be more difficult to deal with with this approach.

The aim of this paper is to give simple procedures for turning a presentation of a locale into a presentation of one of its quotients in a number of important cases. These include open quotients, proper quotients and triquotients, as well as certain `lax' variants of these.

While the results are straightforward applications of the coverage theorems stated below, I have not seen these results mentioned before and I hope that this paper will be a useful reference for dealing with quotients. The last section provides a number of examples to demonstrate the how the results would be used in practice and to showcase their utility.

\section{Background}

We denote the category of frames by $\Frm$ and its opposite category --- the category locales --- by $\Loc$.
To avoid confusion we will distinguish notationally between locales and their corresponding frames: we write $\O X$ for the underlying frame of the locale $X$ and $f^*\colon \O Y \to \O X$ for the frame homomorphism corresponding to the locale map $f\colon X \to Y$.

The category $\Frm$, and thus $\Loc$, has a natural order-enrichment given by the usual pointwise ordering on frame homomorphisms.
Thus it makes sense to consider not only locale coequalisers, but also \emph{coinserters}.
The coinserter of a parallel pair $f,g\colon A \rightrightarrows B$ is the initial map $q\colon B \to C$ such that $qg \le qf$. In this paper we use the convention where the lower path in the diagram is the smaller composite.

A \emph{dcpo} is a poset admitting joins of directed subsets. To emphasise that a join $\bigvee S$ is of a directed subset we we will use the notation $\dirsup S$. The usual morphisms between dcpos preserve these directed joins and are called Scott-continuous functions.
The forgetful functor from $\Frm$ to $\Dcpo$ factors through both the category of suplattices $\Sup$ and the category of preframes $\PreFrm$.
Recall that a suplattice is a poset which admits all joins and a suplattice homomorphism is a join-preserving map between these, while a preframe is a dcpo with finite meets and a preframe homomorphism is a Scott-continuous $\wedge$-semilattice homomorphism.

Of course, frames and suplattices can be presented by generators and relations, but so can preframes (see \cite{johnstone1991preframe}) and dcpos (see \cite{VickersTownsendDoublePowerlocale}).
We write $\langle G \mid R\rangle_{\mathcal{A}}$ for the object of the category $\mathcal{A}$ presented by the generators from $G$ and the relations from $R$ (which we view as a set of formal equalities and inequalities between elements of the free structure on $G$).
For case the $\mathcal{A} = \Frm$ we will omit the subscript when there is no chance of confusion.

Sometimes it is useful to take the generators to have more structure than a mere set. For instance, if $G$ has the structure of a $\wedge$-semilattice, then we can use the free frame on the $\wedge$-semilattice $G$ to form the presentation. This is equivalent to taking the free frame on the set $G$ and adding in relations to force the finite meets in the presented frame to agree with the finite meets in $G$. To indicate we are viewing $G$ as a $\wedge$-semilattice and not just a set we will write such a presentation as $\langle G \text{ $\wedge$-semilattice} \mid R\rangle$.

We now recall some theorems relating different kinds of presentations which will be crucial for proving our results. See \cite{vickers2004double,VickersTownsendDoublePowerlocale} for further details.
\begin{definition}
 Let $G$ be a $\wedge$-semilattice. We will call $\langle G \text{ $\wedge$-semilattice} \mid R\rangle_\Frm$ a \emph{$\Sup$-type frame presentation}
 if every relation in $R$ is of the form $\bigvee A \le \bigvee B$ and furthermore, when $\bigvee A \le \bigvee B$ is a relation, then so is $\bigvee_{a \in A} a \wedge c \le \bigvee_{b \in B} b \wedge c$ for each $c \in G$.
\end{definition}
\begin{theorem}[Suplattice coverage theorem \cite{Abramsky1993quantales}]\label{thm:sup_coverage}
 For a $\Sup$-type frame presentation given by $G$ and $R$, there is an order isomorphism \[\langle G \text{ $\wedge$-semilattice} \mid R\rangle_\Frm \cong \langle G \text{ poset} \mid R\rangle_\Sup.\]
\end{theorem}

\begin{definition}
 Let $G$ be a $\vee$-semilattice. We will call $\langle G \text{ $\vee$-semilattice} \mid R\rangle_\Frm$ a \emph{$\PreFrm$-type frame presentation}
 if every relation in $R$ is of the form $\dirsup[\alpha] \bigwedge A_\alpha \le \dirsup[\beta] \bigwedge B_\beta$ where each $A_\alpha$ and $B_\beta$ is finite, and if furthermore, when $\dirsup[\alpha] \bigwedge A_\alpha \le \dirsup[\beta] \bigwedge B_\beta$ is a relation, then so is $\dirsup[\alpha] \bigwedge_{a \in A_\alpha} (a \vee c) \le \dirsup[\beta] \bigwedge_{b \in B_\beta} (b \vee c)$ for each $c \in G$.
\end{definition}
\begin{theorem}[Preframe coverage theorem \cite{johnstone1991preframe}]\label{thm:prefrm_coverage}
 For a $\PreFrm$-type frame presentation given by $G$ and $R$, there is an order isomorphism \[\langle G \text{ $\vee$-semilattice} \mid R\rangle_\Frm \cong \langle G \text{ poset} \mid R\rangle_\PreFrm.\]
\end{theorem}

\begin{definition}
 Let $G$ be a (bounded) distributive lattice. We will call a presentation $\langle G \text{ dist.\ lattice} \mid R\rangle_\Frm$ a \emph{$\Dcpo$-type frame presentation}
 if every relation in $R$ is of the form $\dirsup A \le \dirsup B$ and furthermore, when $\dirsup A \le \dirsup B$ is a relation, then so are $\dirsup[a \in A] a \wedge c \le \dirsup[b \in B] b \wedge c$ and $\dirsup[a \in A] (a \vee c) \le \dirsup[b \in B] (b \vee c)$ for each $c \in G$.
\end{definition}
\begin{theorem}[Dcpo coverage theorem \cite{VickersTownsendDoublePowerlocale}]\label{thm:dcpo_coverage}
 For a $\Dcpo$-type frame presentation given by $G$ and $R$, there is an order isomorphism \[\langle G \text{ dist.\ lattice} \mid R\rangle_\Frm \cong \langle G \text{ poset} \mid R\rangle_\Dcpo.\]
\end{theorem}

\begin{remark}
 The generators in a $\Sup$-type presentation form a \emph{base} of opens for the topology in question. In a $\PreFrm$-type presentation, it is perhaps most intuitive to think of the elements of the frame, not as being \emph{opens}, but as \emph{closed} sublocales with the reverse order. The generators then give a set of `basic' closed sublocales which are closed under finite meets and generate all the closed sublocales under finite joins and downward-directed meets. For $\Dcpo$-type presentations either viewpoint is appropriate.
\end{remark}

\begin{remark}
The restriction to $\Sup$-type, $\PreFrm$-type or $\Dcpo$-type presentations does not really constrain us, since it is easy to turn any presentation into one of the appropriate form by simply closing the generators under meets or joins as appropriate, manipulating the relations into the appropriate form, and adding any additional relations needed for meet- or join-stability.
\end{remark}

Finally, let us consider the various kinds of maps we will encounter in this paper.
A locale map $f\colon X \to Y$ is said to be \emph{open} if $f^*\colon \O Y \to \O X$ has a left adjoint $f_!\colon \O X \to \O Y$ satisfying the so-called Frobenius condition \[f_!(a \wedge f^*(b)) = f_!(a) \wedge b.\]
This is equivalent to the map $f^*$ being a complete Heyting algebra homomorphism. Note that as a left adjoint $f_!$ preserves joins and is thus a suplattice homomorphism.
We say $f$ is an \emph{open quotient} if $f$ is open and epic --- or equivalently, if $f$ is open and $f^*$ is injective, in which case $f_!$ is a left inverse of $f^*$ in $\Sup$.

A locale map $f\colon X \to Y$ such that $f^*$ has a left adjoint that does not necessarily satisfy the Frobenius condition is called \emph{semi-open}.
We will call epic semi-open maps \emph{semi-open quotients}. These are arguably not true quotient maps, since they need not be regular epimorphisms.
However, they do appear as coinserters, and since many of our results will not need the Frobenius conditions, they are a natural class of maps to condition in this setting.

A map $f\colon X \to Y$ is called \emph{proper} if the right adjoint $f_*$ of $f^*$ is Scott-continuous and satisfies the Frobenius condition \[f_*(a \vee f^*(b)) = f_*(a) \vee b.\]
Note that since right adjoints preserve meets, if $f$ is proper then $f_*$ is a preframe homomorphism.
An epic proper map is called a \emph{proper quotient}. In this case $f_*$ is a left inverse to $f^*$ in $\PreFrm$.
Similarly to above a map $f$ is said to be \emph{semi-proper} (or sometimes \emph{lax proper} or \emph{perfect}) if $f_*$ preserves directed joins, but does not necessary satisfy the Frobenius condition. It will be useful to call epic semi-proper maps \emph{semi-proper quotients}.

A locale map $f\colon X \to Y$ is said to be a \emph{triquotient surjection} if there exists a dcpo morphism $f_\#\colon \O X \to \O Y$, called a \emph{triquotiency assignment},
satisfying $f_\#(a \wedge f^*(b)) = f_\#(a) \wedge b$ and $f_\#(a \vee f^*(b)) = f_\#(a) \vee b$. These generalise open and proper quotients (taking $f_!$ and $f_*$ as the triquotiency assignments respectively). The map $f_\#$ is automatically a left inverse of $f^*$ in $\Dcpo$. See \cite{plewe1997localic} for details.

\section{Working with locale quotients}

If $q\colon X \twoheadrightarrow Y$ is a semi-open quotient, then the composite of $q^*q_!$ gives a suplattice endomorphism on $\O X$.
In fact, this is a closure operator and $\O Y$ is isomorphic to its fixed points. Moreover, the fixed points of any join-preserving closure operator on $\O X$ form a frame (since they are closed under all meets and joins) and give rise to a semi-open quotient. Thus, these will be a useful way for us to specify the quotient locale $Y$ without already knowing the precise form of $Y$.
The map $q$ will be \emph{open} if and only if its corresponding closure operator $j$ satisfies $j(a) \wedge j(b) \le j(a \wedge j(b))$.

If $q\colon X \twoheadrightarrow Y$ is a semi-proper quotient, then $q^*q_*$ similarly gives a preframe endomorphism and interior operator on $\O X$. As before, $\O Y$ is isomorphic to its fixed points and the fixed points of any such operator form a frame and give a semi-proper quotient. The quotient is proper if and only if the interior operator satisfies $p(a) \vee p(b) \ge p(a \vee p(b))$.

It is shown in \cite{kock1989godement} that if $p_1,p_2\colon R \rightrightarrows X$ are open morphisms and are the domain and codomain maps of a localic groupoid (or in particular, an equivalence relation), then their coequaliser $q\colon X \twoheadrightarrow Y$ is an open quotient and $\O Y$ is given by the fixed points of the closure operator $(p_1)_!p_2^*$ (or $(p_2)_!p_1^*$). Furthermore, even if only the upper (domain) map $p_1$ map is open and they form a localic category (or in particular, a preorder), the coinserter is a semi-open quotient and the corresponding closure operator is still $(p_1)_!p_2^*$. 
Dual results with proper maps replacing open maps are given in \cite{vermeulen1994proper,korostenski2007lax}. (There it is only stated for internal equivalence relations and preorders, but the proofs work equally well for localic groupoids and categories.) This time the \emph{lower} map $p_2$ should be proper and $\O Y$ is given by the fixed points of the interior operator $(p_2)_*p_1^*$.

In fact, we can obtain similar results in a more general situation. Suppose the following diagram is a coinserter in $\Loc$.
\begin{center}
\begin{tikzpicture}
 \node (A) {$R$};
 \node [right=0.6cm of A] (B) {$X$};
 \node [right=0.6cm of B] (C) {$Y$};
 \draw[transform canvas={yshift=0.5ex},->] (A) to node [above] {$f$} (B);
 \draw[transform canvas={yshift=-0.5ex},->] (A) to node [below] {$g$} (B);
 \draw[->>] (B) to node [above] {$q$} (C);
\end{tikzpicture}
\end{center}
Since the forgetful functor from $\Frm$ to $\Pos$ creates weighted limits, $\O Y$ can be identified with $\{ u \in \O X \mid g^*(u) \le f^*(u) \}$.
If $f$ is semi-open, then $g^*(u) \le f^*(u) \iff f_!g^*(u) \le u$ and so $\O Y$ consists of the \emph{pre-fixed points} of the suplattice endomorphism $f_!g^*$ on $\O X$.
Similarly, if $g$ is semi-proper, then $\O Y$ consists of the \emph{post-fixed points} of the preframe endomorphism $g_*f^*$.

\begin{remark}
 If the above is a \emph{reflexive} coinserter (with common section $r$), then $g_*f^* = r^*g^*g_*f^* \le r^*f^* = \id$ and so $g_*f^*$ is automatically deflationary. Similarly, if it exists,
 $f_!g^*$ is automatically inflationary. So in this case the coinserter is given by the \emph{fixed points} of $f_!g^*$ or $g^*f_*$, not just the pre-fixed points or post-fixed points.
\end{remark}

\begin{proposition}\label{prop:open_quotient_endomorphism}
If $j$ is a suplattice endomorphism on a frame $\O X$, then the pre-fixed points of $j$ form a subframe $\O Y$ of $\O X$. The frame inclusion $q^*\colon \O Y \hookrightarrow \O Y$ has a left adjoint and thus corresponds to a semi-open locale quotient. The associated closure operator is given by $\bigvee_{n = 0}^\infty j^n$.

Consequently, given a coinserter diagram as above, if $f$ is semi-open then so is $q$.
\end{proposition}
\begin{proof}
 The map $j \vee \id$ is an inflationary suplattice endomorphism, the fixed points of which coincide with the pre-fixed points of $j$.
 As in the Kleene fixed point theorem, we find the map $\bigvee_{n = 0}^\infty j^n$ is a join-preserving closure operator, which again has the same fixed points.
 This provides a left adjoint to the inclusion $q^*\colon \O Y \hookrightarrow \O X$.
 
 Now consider the coinserter where $f$ is semi-open. The result follows by applying the above to the suplattice endomorphism $f_!g^*$.
\end{proof}

\begin{remark}\label{rem:open_coequaliser_and_symmetry}
 Since $g^*(u) = f^*(u)$ whenever $g^*(u) \le f^*(u)$ and $f^*(u) \le g^*(u)$, it follows that if both $f$ and $g$ are semi-open then their \emph{coequaliser} is given by the pre-fixed points of $f_!g^* \vee g_!f^*$.
 Furthermore, if there is a `symmetry' map $s\colon R \to R$ such that $g = fs$ (as in the case of a localic groupoid) then it can be shown that the coequaliser and coinserter coincide.
\end{remark}
\begin{remark}
 The fact that the coequaliser of semi-open maps is semi-open is also immediate from the fact that equalisers in the category of complete lattices and agree with those in $\Frm$ (since they are both computed as in $\Set$).
 Similarly, we see that the coequaliser of open maps is open by taking the equaliser in the category of complete Heyting algebras.
\end{remark}

The case of proper quotients is somewhat more subtle.
\begin{example}
 Let $N$ be $\N$ equipped with the order topology for the reverse of the usual order (i.e.\ opens are downsets with respect to the usual order)
 and let $s\colon N \to N$ be the successor function. This map is proper.
 Then the coequaliser and the coinserter of $s$ and $\id_N$ (with $\id_N$ the lower map) are both the terminal locale.
 But the unique map ${!}\colon N \to 1$ is proper if and only if it is semi-proper if and only if $N$ is compact and it is clear that $N$ is not compact.
 Thus, a coequaliser or coinserter of proper maps need not even be semi-proper.
\end{example}

The problem is that we cannot iterate the preframe endomorphism to obtain an idempotent one as we did for the suplattice endomorphism in \cref{prop:open_quotient_endomorphism}.
The following proposition gives conditions under which idempotence is automatic and so we do obtain a semi-proper quotient.

\begin{proposition}\label{prop:proper_transitive_interior_operator}
 Consider a coinserter diagram as above and suppose $g$ is \emph{proper}. Further suppose that there is a locale map $t\colon R\times_X R \to R$ such that $g \pi_2 \le gt$ and $ft \le f \pi_1$ (for example, the transitivity map of an internal preorder or the composition of an internal category).
 Then the coinserter morphism $q$ is semi-proper and $g_*f^* \wedge \id$ is the associated interior operator.
 \begin{center}
  \begin{tikzpicture}[node distance=2.5cm, auto]
    \node (A) {$R \times_X R$};
    \node (B) [below of=A] {$R$};
    \node (B') [below=1.25cm of B] {$X$};
    \node (C) [right of=A] {$R$};
    \node (C') [right=1.25cm of C] {$X$};
    \node (D) [below of=C] {$X$};
    \draw[->] (A) to node [swap] {$\pi_1$} (B);
    \draw[->] (A) to node {$\pi_2$} (C);
    \draw[->] (B) to node [swap] {$g$} (D);
    \draw[->] (C) to node {$f$} (D);
    \draw[->] (B) to node [swap] {$f$} (B');
    \draw[->] (C) to node {$g$} (C');
    
    \begin{scope}[shift=({A})]
        \draw +(0.25,-0.75) -- +(0.75,-0.75) -- +(0.75,-0.25);
    \end{scope}
    
    \node (X) [above left=1.2cm and 1.2cm of A.center] {$R$};
    \draw[out=-105,->] (X) to node [swap] {$f$} node [anchor=center] (F) {} (B');
    \draw[out=15,->] (X) to node {$g$} node [anchor=center] (G) {} (C');
    \draw[<-, pos=0.55] (X) to node {$t$} (A);
    
    \path[pos=0.6] (A) to node [anchor=center] {$\le$} (F);
    \path[pos=0.6] (A) to node [anchor=center] {\rotatebox{90}{$\le$}} (G);
  \end{tikzpicture}
 \end{center}
\end{proposition}
\begin{proof}
 Since $g$ is proper, so is the pullback projection $\pi_2$. Moreover, we have the Beck--Chevalley condition: $f^*g_* = (\pi_2)_* \pi_1^*$.
 Thus,
 \begin{align*}
  g_*f^*g_*f^* &= g_* (\pi_2)_* \pi_1^* f^* \\
               &\ge g_* (\pi_2)_* t^* f^* \\
               &\ge g_* (\pi_2)_* t^* g^*g_* f^* \\
               &\ge g_* (\pi_2)_* \pi_2^* g^*g_* f^* \\
               &\ge g_* g^* g_*f^* \\
               &= g_*f^*.
 \end{align*}
 Now $g_*f^* \wedge \id$ is a deflationary preframe endomorphism on $\O X$ and by the above we have $(g_*f^* \wedge \id)^2 = g_*f^*g_*f^* \wedge g_*f^* \wedge \id = g_*f^* \wedge \id$
 and hence $g_*f^* \wedge \id$ is an interior operator.
 Finally, the elements of $\O Y$ are precisely the post-fixed points of $g_*f^*$ and thus the fixed points of $g_*f^* \wedge \id$.
\end{proof}
\begin{remark}
 If $g$ is only semi-proper, then a similar result holds if the domain of $t$ is replaced with the comma object $g/f$.
\end{remark}

\begin{remark}\label{rem:proper_coequaliser_and_symmetry}
 If in the above proposition $f$ and $g$ are both proper, $g\pi_2 = gt$ and $f\pi_1 = ft$, then their \emph{coequaliser} is given by the fixed points of the interior operator $g_*f^* \wedge f_*g^* \wedge \id$.
 Moreover, the resulting quotient map is proper.
 Finally, if there is a symmetry map as in \cref{rem:open_coequaliser_and_symmetry} the coequaliser and coinserter coincide.
\end{remark}

More general than proper and open quotients is the case of triquotients. While there is unfortunately no good result relating these to coequalisers, we can still represent them by idempotent dcpo endomorphisms in analogy to the interior and closure operators discussed above.
In fact, just as proper and open quotients can be generalised to semi-proper and semi-open quotients, we consider general locale maps $q\colon X \twoheadrightarrow Y$ whose corresponding frame map $q^*$ has a Scott-continuous retraction $q_\#$, but without any additional assumptions. These do not seem to have an established name, but one might call them \emph{semi-triquotient} maps or perhaps even \emph{sesqui-quotient} maps.

If $q\colon X \twoheadrightarrow Y$ is a semi-triquotient, the map $q^*q_\#$ is an idempotent dcpo endomorphism on $\O X$, whose poset of fixed points is isomorphic to $\O Y$. This satisfies $q^*q_\#(1) = 1$, $q^*q_\#(0) = 0$, $q^*q_\#(a) \wedge q^*q_\#(b) \le q^*q_\#(q^*q_\#(a) \wedge q^*q_\#(b))$ and $q^*q_\#(a) \vee q^*q_\#(b) \ge q^*q_\#(q^*q_\#(a) \vee q^*q_\#(b))$. Moreover, the fixed points of any dcpo endomorphism $e$ on a frame $\O X$ satisfying these conditions form a subframe of $\O X$ whose inclusion corresponds to a semi-triquotient with dcpo retraction $e$.

An idempotent dcpo endomorphism $e$ corresponds to a \emph{triquotient} if it satisfies the stronger laws $e(a) \wedge e(b) \le e(a \wedge e(b))$ and $e(a) \vee e(b) \ge e(a \vee e(b))$, in addition to $e(1) = 1$ and $e(0) = 0$.

Finally, we note that by general categorical principles in all the cases above the frame of fixed points of the suplattice, preframe or dcpo endomorphism describing the quotient locale can be found not only as a subobject, but also a quotient in the appropriate category. For example, the map $q_\#\colon \O X \twoheadrightarrow \O Y$ is the dcpo coequaliser of $q^*q_\#$ and $\id_{\O X}$. Thus, $\O Y$ can be obtained as a dcpo quotient of $\O X$ by setting $a \sim q^*q_\#(a)$ for each $a$ in $\O X$ (or equivalently, for each $a$ in some subset that generates $\O X$ under directed joins).

Similarly, if $j\colon \O X \to \O X$ is a join-preserving closure operator, the suplattice quotient onto the frame of fixed points is obtained by setting $j(a) \lesssim a$ for each $a$ in some base of $\O X$. If $p\colon \O X \to \O X$ is a preframe endomorphism and an interior operator, the preframe quotient is obtained by setting $a \lesssim p(a)$ for each $a$ in some subset of $\O X$ which generates $\O X$ as a preframe.

\section{Main results}

While presenting quotients of locales is tricky in general, there is one case where it is trivial. This is when the quotient $q\colon X \twoheadrightarrow Y$ has a section $s\colon Y \hookrightarrow X$.
Then $Y$ is a sublocale of $X$ and so we simply need to add the additional corresponding additional relations to the presentation of $X$ in the usual way.

In the cases we will consider, we will not be quite so lucky to have a frame homomorphism which is left inverse to $q^*$, but we will instead have morphisms of suplattices, preframes or dcpos that will play the same role.

\subsection{Presenting open quotients}

We know that a semi-open quotient $q\colon X \twoheadrightarrow Y$ can be specified by a join-preserving closure operator $j\colon \O X \to \O X$
and that the map $q_!\colon \O X \to \O Y$ is a suplattice quotient with kernel congruence generated by $j(a) \lesssim a$.

Now suppose we have a presentation for $\O X$.
Without loss of generality, we may assume that this is a $\Sup$-type presentation $\O X = \langle G \text{ $\wedge$-semilattice} \mid R \rangle_\Frm$.
We can then use the suplattice coverage theorem (\ref{thm:sup_coverage}) to obtain $\O X \cong \langle G \text{ poset} \mid R \rangle_\Sup$.

Adding the relations from the suplattice quotient given by $q_!$ we obtain a suplattice presentation for $\O Y$. 
We find $\O Y \cong \langle G \text{ poset} \mid R,\ j(g) \le g,\, g \in G \rangle_\Sup$.
Note that here the relations $j(g) \le g$ can be understood purely in terms of generators by writing each $j(g)$ as a join of generators in $\O X$.

We still need to turn this into a \emph{frame} presentation of $\O Y$.
In general if a suplattice $L$ happens to be a frame, then $L \cong \langle \{\lozenge a \mid a \in L\} \text{ suplattice} \mid \lozenge 1 = 1,\ \lozenge a \wedge \lozenge b = \lozenge (a \wedge b),\, a, b \in L \rangle_\Frm$, where the meet $a \wedge b$ is taken in $L$. In fact, it suffices to only include (in addition to $\lozenge 1 = 1$) the relations $\lozenge a \wedge \lozenge b = \lozenge (a \wedge b)$ for $a$ and $b$ restricted to some base of $L$, since the remaining relations will follow from these ones by taking joins and using the fact that $\lozenge (-)$ preserves the joins.
In our case generators $a,b$ in $\O X$ map to $q_!(a), q_!(b)$ in the suplattice quotient and so taking the meet in $\O Y$ corresponds to $j(a) \wedge j(b)$ in our presentation.
Thus, combining this with the above suplattice presentation, we have
\begin{align*}
 \O Y \cong \langle \{ \lozenge g \mid  g \in G\} \text{ poset} \mid {}
     & R,\ \lozenge j(g) \le \lozenge g,\, g \in G,\ \lozenge 1 = 1,\\
     & \lozenge s \wedge \lozenge t = \lozenge(j(s) \wedge j(t)),\, s, t \in G \rangle_\Frm,
\end{align*}
where $\lozenge j(g)$ is defined to mean $\bigvee_\alpha \lozenge g_\alpha$ for $j(g) = \bigvee_\alpha g_\alpha$ and similarly, we interpret $\lozenge (j(s) \wedge j(t)) = \bigvee_{\sigma,\tau} \lozenge (g_\sigma \wedge g_\tau)$ for $j(s) = \bigvee_\sigma g_\sigma$ and $j(t) = \bigvee_\tau g_\tau$ (where the meet is taken in the $\wedge$-semilattice $G$).

Finally, note that the relation $\lozenge j(g) \le \lozenge g$ follows from the meet conditions by taking $s = 1$ and $t = g$,
and the inequalities from the poset of generators then also follow from the meet condition by taking $s \le t$.
We have arrived at the following result.

\begin{proposition}\label{prop:semiopen_quotient}
 Suppose $\O X = \langle G \text{ $\wedge$-semilattice} \mid R \rangle_\Frm$ is a $\Sup$-type presentation and let $q\colon X \twoheadrightarrow Y$ be a semi-open quotient.
 Then
 \begin{align*}
 \O Y \cong \langle \lozenge g,\, g \in G \mid {}
     & R,\ \lozenge 1 = 1, \\
     & \lozenge s \wedge \lozenge t = \lozenge (q^*q_!(s) \wedge q^*q_!(t)),\, s, t \in G \rangle_\Frm,
\end{align*}
where we interpret $\lozenge (q^*q_!(s) \wedge q^*q_!(t)) = \bigvee_{\alpha,\beta} \lozenge (s_\alpha \wedge t_\beta)$ for specified representations $q^*q_!(s) = \bigvee_\alpha s_\alpha$ and $q^*q_!(t) = \bigvee_\beta t_\beta$ in terms of generators.
\end{proposition}

If $q$ is an \emph{open} quotient then we can simply this further still.
\begin{corollary}\label{prop:open_quotient}
 Suppose $\O X = \langle G \text{ $\wedge$-semilattice} \mid R \rangle_\Frm$ is a $\Sup$-type presentation and let $q\colon X \twoheadrightarrow Y$ be an open quotient.
 Then
 \begin{align*}
 \O Y \cong \langle \lozenge g,\, g \in G \mid {}
     & R,\ \lozenge 1 = 1, \\
     & \lozenge s \wedge \lozenge t = \lozenge (s \wedge q^*q_!(t)),\, s, t \in G \rangle_\Frm,
\end{align*}
where we interpret $\lozenge (s \wedge q^*q_!(t)) = \bigvee_\beta \lozenge (s \wedge t_\beta)$ for specified representation $q^*q_!(t) = \bigvee_\beta t_\beta$.
\end{corollary}
\begin{proof}
 Again, taking $s = 1$ and $t = g$ gives the relation $\lozenge g = \lozenge q^*q_!(g)$.
 Then the Frobenius condition gives $q^*q_!(s) \wedge q^*q_!(t) = q^*q_!(s \wedge q^*q_!(t))$.
 Combining these we obtain the meet relations given in \cref{prop:semiopen_quotient} and the converse direction is straightforward.
\end{proof}

\begin{remark}
 Our derived presentation can be understood as describing the quotient locale $Y$ as a sublocale of the \emph{lower powerlocale} of $X$.
 The frame of opens of the lower powerlocale is the free frame on the underlying suplattice of $\O X$ and its points correspond to (overt, weakly) closed sublocales of $X$.
 We can identify points of $Y$ with the sublocales of $X$ that appear as the (weak) closures of their fibres under $q\colon X \to Y$. Intuitively, we can view it as a space of equivalence classes.
 See \cite{Vickers1997PowerlocalePoints,vickers1994pointless} for more details on the lower powerlocale and its relation to semi-open maps.
\end{remark}

\subsection{Presenting proper quotients}

We can proceed similarly for semi-proper quotients.
Recall that a semi-proper quotient $q\colon X \twoheadrightarrow Y$ can be specified by an interior operator and preframe endomorphism $p\colon \O X \to \O X$.
The map $q_*$ is the preframe quotient given by setting $a \lesssim p(a)$.

Suppose $\O X$ has a $\PreFrm$-type presentation $\langle G \text{ $\vee$-semilattice} \mid R \rangle_\Frm$.
Applying the preframe coverage theorem (\ref{thm:prefrm_coverage}) and taking the quotient we have $\O Y \cong \langle G \text{ poset} \mid R,\ g \le p(g),\, g \in G \rangle_\PreFrm$.
Here we expand each $p(g)$ as a directed join of finite meets of generators.

Now we need to turn this into a frame presentation.
Similarly to before, if a preframe $L$ happens to be a frame, then $L \cong \langle \{\square a \mid a \in L\} \text{ preframe} \mid \square 0 = 0,\ \square (a \vee b) = \square a \vee \square b,\, a,b \in L \rangle_\Frm$.
We may also restrict $a,b$ to lie in a preframe generating set.
In our case, the generators $a,b$ in $\O X$ map to $q_*(a), q_*(b)$ in the preframe quotient and so their join in $\O Y$ corresponds to $p(a) \vee p(b)$ in $\O X$.
Combining this with the preframe presentation above and eliminating redundant relations as before we arrive as the following result.

\begin{proposition}\label{prop:semiproper_quotient}
 Suppose $\O X = \langle G \text{ $\vee$-semilattice} \mid R \rangle_\Frm$ is a $\PreFrm$-type presentation and let $q\colon X \twoheadrightarrow Y$ be a semi-proper quotient.
 Then
 \begin{align*}
 \O Y \cong \langle \square g,\, g \in G \mid {}
     & R,\ \square 0 = 0 \\
     & \square s \vee \square t = \square (q^* q_*(s) \vee q^* q_*(t)),\, s, t \in G \rangle_\Frm,
 \end{align*}
 where $\square (q^* q_*(s) \vee q^* q_*(t)) = \dirsup[\alpha,\beta] \bigwedge_{i_\alpha,j_\beta} \square (s_\alpha^{i_\alpha} \vee t_\beta^{j_\beta})$ for specified representations $q^*q_!(s) = \dirsup[\alpha] \bigwedge_{i_\alpha} s_\alpha^{i_\alpha}$ and $q^*q_!(t) = \dirsup[\beta] \bigwedge_{j_\beta} t_\beta^{j_\beta}$ in terms of generators.
\end{proposition}

If $q$ is an \emph{proper} quotient we can again use the Frobenius condition to give the following simplification.
\begin{corollary}\label{prop:proper_quotient}
 Suppose $\O X = \langle G \text{ $\vee$-semilattice} \mid R \rangle_\Frm$ is a $\PreFrm$-type presentation and let $q\colon X \twoheadrightarrow Y$ be a proper quotient.
 Then
 \begin{align*}
 \O Y \cong \langle \square g,\, g \in G \mid {}
     & R,\ \square 0 = 0 \\
     & \square s \vee \square t = \square (s \vee q^* q_*(t)),\, s, t \in G \rangle_\Frm,
 \end{align*}
 where $\square (s \vee q^* q_*(t)) = \dirsup[\beta] \bigwedge_{j_\beta} \square (s \vee t_\beta^{j_\beta})$ for specified representation $q^*q_!(t) = \dirsup[\beta] \bigwedge_{j_\beta} t_\beta^{j_\beta}$.
\end{corollary}

\begin{remark}
 This time our presentation can be thought of as expressing $Y$ as a sublocale of the \emph{upper powerlocale} of $X$.
 The frame of opens of the upper powerlocale is the free frame on the underlying preframe of $\O X$ and its points correspond to compact fitted sublocales of $X$.
 As before, we can think of $Y$ as a space of equivalence classes: the points of $Y$ correspond to the fitting of the fibres of $q\colon X \to Y$.
 Again see \cite{Vickers1997PowerlocalePoints,vickers1994pointless} for more details on these concepts.
 \Cref{prop:semiopen_quotient,prop:semiproper_quotient} can also be compared to the construction given in the proof of (ii) $\Rightarrow$ (iii) of \cite[Theorem 2]{townsend2005categorical}.
\end{remark}

\subsection{Presenting triquotient locales}

The general semi-triquotient case is again similar. Let $q\colon X \twoheadrightarrow Y$ be a locale map such that $q^*$ has a left inverse dcpo morphism $q_\#\colon \O X \to \O Y$. Such a semi-triquotient can represented by a dcpo endomorphism $e$ satisfying the necessary conditions and the retraction onto the fixed points induced by $e$ corresponds the dcpo quotient map $q_\#$, which is specified by setting $a \sim e(a)$.

Now suppose $\O X$ has a $\Dcpo$-type presentation $\langle G \text{ dist.\ lattice} \mid R\rangle_\Frm$. Using the dcpo coverage theorem (\ref{thm:dcpo_coverage}) and taking the quotient we have $\O Y \cong \langle G \text{ poset} \mid R,\ g = e(g),\, g \in G \rangle_\Dcpo$, where each $e(g)$ is expressed explicitly as a directed join of generators.
Then as before we can turn this into a frame presentation to obtain the following.

\begin{proposition}\label{prop:semitriquotient}
 Suppose $\O X = \langle G \text{ dist.\ lattice} \mid R \rangle_\Frm$ is a $\Dcpo$-type presentation and let $q\colon X \twoheadrightarrow Y$ be a semi-triquotient with triquotiency assignment $q_\#\colon \O X \to \O Y$.
 Then
 \begin{align*}
 \O Y \cong \langle {\boxtimes} g,\, g \in G \mid {}
     & R,\ {\boxtimes} 1 = 1,\ {\boxtimes} 0 = 0 \\
     & {\boxtimes} s \wedge {\boxtimes} t = {\boxtimes} (q^*q_\#(s) \wedge q^*q_\#(t)), \\
     & {\boxtimes} s \vee {\boxtimes} t = {\boxtimes} (q^* q_\#(s) \vee q^* q_\#(t)),\ s, t \in G \rangle_\Frm,
 \end{align*}
 where ${\boxtimes} (q^*q_\#(s) \wedge q^*q_\#(t)) = \dirsup[\alpha,\beta] {\boxtimes} (s_\alpha \wedge t_\beta)$
 and ${\boxtimes} (q^* q_\#(s) \vee q^* q_\#(t)) = \dirsup[\alpha,\beta] {\boxtimes} (s_\alpha \vee t_\beta)$ for specified representations $q^*q_\#(s) = \dirsup[\alpha] s_\alpha$ and $q^*q_\#(t) = \dirsup[\beta] t_\beta$ in terms of generators.
\end{proposition}

Finally, if $q$ is an \emph{triquotient} we can again simplify things a little.
\begin{corollary}\label{prop:triquotient}
 Suppose $\O X = \langle G \text{ dist.\ lattice} \mid R \rangle_\Frm$ is a $\Dcpo$-type presentation and let $q\colon X \twoheadrightarrow Y$ be a triquotient with triquotiency assignment $q_\#\colon \O X \to \O Y$.
 Then
 \begin{align*}
 \O Y \cong \langle {\boxtimes} g,\, g \in G \mid {}
     & R,\ {\boxtimes} 1 = 1,\ {\boxtimes} 0 = 0 \\
     & {\boxtimes} s \wedge {\boxtimes} t = {\boxtimes} (s \wedge q^*q_\#(t)), \\
     & {\boxtimes} s \vee {\boxtimes} t = {\boxtimes} (s \vee q^* q_\#(t)),\ s, t \in G \rangle_\Frm,
 \end{align*}
 where ${\boxtimes} (s \wedge q^*q_\#(t)) = \dirsup[\beta] {\boxtimes} (s \wedge t_\beta)$
 and ${\boxtimes} (s \vee q^* q_\#(t)) = \dirsup[\beta] {\boxtimes} (s \vee t_\beta)$ for specified representation $q^*q_\#(t) = \dirsup[\beta] t_\beta$.
\end{corollary}

\begin{remark}
 While these (semi-)triquotient results technically subsume the previous ones, they are all useful, since the presentations obtained from the more specific results will be simpler those given by the most general one.
\end{remark}

\begin{remark}
 This time the intuition is less clear, but the presentation is related to viewing $Y$ is a sublocale of the \emph{double powerlocale} of $X$ (see \cite{vickers2004double}).
 The double powerlocale arises from the adjunction between frames and dcpos is equal to the composition of the upper and lower powerlocales (in either order). Its points can be viewed as certain overt collections of compact sublocales.
 
 It is natural to ask if yet more general quotients might arise by using further composites of upper and lower powerlocales, but essentially because the double powerlocale is a retract of the `quadruple powerlocale', no new quotients are obtained after the second level.
\end{remark}

\section{Examples and applications}

The main utility of these results is in deriving concrete presentations, but they do have at least one theoretical consequence:
namely, that the size of the presentation of the quotient is not so different from that of the presentation of the parent locale.
In particular, we can immediately deduce the following.
\begin{proposition}
 A semi-triquotient of a countably presented locale is countably presented.
\end{proposition}

Classically, countably presented locales coincide with \emph{quasi-Polish} spaces (see \cite{debrecht2013quasipolish,heckmann2015}).
Thus, as a corollary we obtain the following result, which known for open quotients, but which I have not seen stated before at this level of generality.
\begin{corollary}
 Let $X$ be a quasi-Polish space, let $Y$ be sober and let $X \twoheadrightarrow Y$ be a triquotient map. Then $Y$ is quasi-Polish.
\end{corollary}

We now conclude with some illustrative examples. First we obtain a presentation of the circle $\T$ from the presentation of $\R$ using the open quotient $\R \twoheadrightarrow \T$.

\begin{example}
The usual presentation for the frame of reals $\O\R$ has a generator $\llround q, \infty\rrround$ and a generator $\llround -\infty, q\rrround$ for each $q \in \Q$. They satisfy the following relations.
\begin{itemize}
 \item $\llround p, \infty\rrround \wedge \llround -\infty, q\rrround = 0$ for $p \ge q$,
 \item $\llround p, \infty\rrround \vee \llround -\infty, q\rrround = 1$ for $p < q$,
 \item $\llround q, \infty\rrround = \bigvee_{p > q} \llround p, \infty\rrround$ and $\llround -\infty, q\rrround = \bigvee_{p < q} \llround -\infty, p\rrround$ for all $q \in \Q$,
 \item $\bigvee_{q \in \Q} \llround q, \infty\rrround = 1$ and $\bigvee_{q \in \Q} \llround -\infty, q\rrround = 1$.
\end{itemize}

To apply our result the presentation must be modified to be of the appropriate form.
In particular, the generators should be closed under finite meets. Setting $\llround p, q\rrround = \llround p, \infty\rrround \wedge \llround -\infty, q\rrround$ and $\llround -\infty, \infty \rrround = 1$
we have defined generators $\llround p, q\rrround$ for each $p \in \Q \sqcup \{-\infty\}$ and $q \in \Q \sqcup \{\infty\}$. These form a $\wedge$-semilattice with $\llround p, q \rrround \wedge \llround p', q' \rrround = \llround p \vee p', q \wedge q'\rrround$.
The remaining axioms become:
\begin{itemize}
 \item $\llround p,q\rrround = 0$ for $p \ge q$,
 \item $\llround p,q\rrround \vee \llround p',q'\rrround = \llround p,q'\rrround$ for $p \le p' < q \le q'$,
 \item $\llround p,q\rrround = \bigvee_{p < p' < q' < q} \llround p',q'\rrround$ for all $p \in \Q \sqcup \{-\infty\}$ and $q \in \Q \sqcup \{\infty\}$.
\end{itemize}
In fact, the first bullet point here is redundant, since it follows from the third.
We note that these relations (essentially) satisfy the necessary meet-stability conditions for this to be a $\Sup$-type presentation.
Let us show this for a representative case.

Since $\llround a,b \rrround = \llround a, \infty \rrround \wedge \llround -\infty,b \rrround$, it suffices to only consider meets with $\llround a, \infty \rrround$ and $\llround -\infty,b \rrround$.
Consider the relation $\llround p,q\rrround \vee \llround p',q'\rrround = \llround p,q'\rrround$ for $p \le p' < q \le q'$.
Meeting with $\llround a, \infty \rrround$ we obtain the equality $\llround p \vee a, q\rrround \vee \llround p' \vee a,q'\rrround = \llround p \vee a,q'\rrround$.
Meet stability demands that this is also a relation in the presentation. If $a \le p'$ then $p \vee a \le p' \vee a < q \le q$ and so this is indeed an assumed relation of the same form.
If not, then $p' < a$ and it becomes $\llround a, q\rrround \vee \llround a,q'\rrround = \llround a,q'\rrround$.
Technically, this is not one of the original relations, but it is trivial in the sense that it always holds.

Indeed, even more generally, our results will still hold as long as the desired relations lie in the suplattice congruence generated from the core relations and the order on the generators, since we can add these relations in to achieve meet stability, but then omit them from the final presentation (since they follow from the others by assumption).

Meeting with $\llround -\infty, b \rrround$ works similarly and meet stability for the final relations can also be shown to hold, at least in this weaker sense.

Now consider the coequaliser
\begin{center}
\begin{tikzpicture}
 \node (A) {$\R$};
 \node [right=1.0cm of A] (B) {$\R$};
 \node [right=0.8cm of B] (C) {$\T$.};
 \draw[transform canvas={yshift=0.5ex},->] (A) to node [above] {$\id$} (B);
 \draw[transform canvas={yshift=-0.5ex},->] (A) to node [below] {$+1$} (B);
 \draw[->>] (B) to node [above] {$q$} (C);
\end{tikzpicture}
\end{center}
This is an open coequaliser, since the both of the parallel arrows are isomorphisms.
So by \cref{prop:open_quotient_endomorphism,rem:open_coequaliser_and_symmetry}, the corresponding closure operator $q^*q_!$ is given by
$\bigvee_{n \in \N } (\id_! \circ (+1)^*)^n \vee \bigvee_{n \in \N } ((+1)_! \circ \id^*)^n$, which equals $\bigvee_{n \in \Z} (+n)^*$, since the left adjoint of $(+1)^*$ is simply its inverse.
Note that $(+1)^*\llround p, q\rrround = \llround p-1, q-1\rrround$ and so in terms of generators, we have $q^*q_! \llround p, q\rrround = \bigvee_{n \in \Z} \llround p + n, q + n\rrround$.

We can now use \cref{prop:open_quotient} to obtain a presentation for $\O\T$ with generators $\llround p, q\rrround$ for $p \in \Q \sqcup \{-\infty\}$ and $q \in \Q \sqcup \{\infty\}$ and the following relations:
\begin{itemize}
 \item $\llround -\infty, \infty \rrround = 1$,
 \item $\llround p, q \rrround \wedge \llround p', q' \rrround = \bigvee_{n \in \Z} \llround p \vee (p' + n), q \wedge (q' + n)\rrround$,
 \item $\llround p,q\rrround \vee \llround p',q'\rrround = \llround p,q'\rrround$ for $p \le p' < q \le q'$,
 \item $\llround p,q\rrround = \bigvee_{p < p' < q' < q} \llround p',q'\rrround$.
\end{itemize}

This presentation for the circle is essentially the same as the one in \cite[Section 5]{gutierrezgarcia2016circle} which was obtained by ad hoc methods. In contrast, we observe that our derivation was very natural and mechanistic.
\end{example}

For our next example, we will deduce a different presentation for $\T$ from the proper quotient $[0,1] \twoheadrightarrow \T$.

\begin{example}
A standard presentation for the frame $\O [0,1]$ has generators $\llround q, 1\rrbracket$ and $\llbracket 0, q\rrround$ for each $q \in \Q \cap [0,1]$ which satisfy the relations:
\begin{itemize}
 \item $\llround p, 1\rrbracket \wedge \llbracket 0, q\rrround = 0$ for $p \ge q$,
 \item $\llround p, 1\rrbracket \vee \llbracket 0, q\rrround = 1$ for $p < q$,
 \item $\llround q, 1\rrbracket = \bigvee_{p > q} \llround p, 1\rrbracket$ and $\llbracket 0, q\rrround = \bigvee_{p < q} \llbracket 0, p\rrround$ for all $q \in \Q \cap [0,1]$.
\end{itemize}

To use our result for proper quotients this presentation must also be modified. This time we want the generators to be closed under finite joins.
We set ${\rrparenthesis p,q \llparenthesis} = \llbracket 0, p\rrround \vee \llround q, 1\rrbracket$. Intuitively, these are the complements of closed intervals.

Indeed, recall that for $\PreFrm$-type presentations the elements of the frame are best thought of as \emph{closed} sublocales under the reverse order, so we can imagine the generator ${\rrparenthesis p,q \llparenthesis}$ as corresponding to the closed interval $[p, q]$. From this perspective it is intuitive that under intersections of the corresponding closed intervals these form a very similar semilattice to the formal open intervals considered above.
Formally, we can check that these indeed form a $\vee$-semilattice with ${\rrparenthesis p,q \llparenthesis} \vee {\rrparenthesis p',q' \llparenthesis} = {\rrparenthesis p \vee p', q \wedge q' \llparenthesis}$ and ${\rrparenthesis 0, 1 \llparenthesis} = 0$.

In terms of these new generators the remaining axioms become:
\begin{itemize}
 \item ${\rrparenthesis p,q \llparenthesis} \wedge {\rrparenthesis p',q' \llparenthesis} = {\rrparenthesis p,q' \llparenthesis}$ for $p \le p' \le q \le q'$,
 \item ${\rrparenthesis p,q \llparenthesis} = 1$ for $p > q$,
 \item ${\rrparenthesis p,q \llparenthesis} = \dirsup[q' > q] {\rrparenthesis p,q' \llparenthesis}$ for $p, q < 1$ and ${\rrparenthesis p,q \llparenthesis} = \dirsup[p' < p] {\rrparenthesis p',q \llparenthesis}$ for $0 < p, q$.
\end{itemize}
Again note that these essentially satisfy the join-stability conditions for these to form a $\PreFrm$-style presentation.

We can now consider the locale coequaliser
\begin{center}
\begin{tikzpicture}
 \node (A) {$1$};
 \node [right=1.0cm of A] (B) {$[0,1]$};
 \node [right=0.8cm of B] (C) {$\T$.};
 \draw[transform canvas={yshift=0.5ex},->] (A) to node [above] {$0$} (B);
 \draw[transform canvas={yshift=-0.5ex},->] (A) to node [below] {$1$} (B);
 \draw[->>] (B) to node [above] {$q$} (C);
\end{tikzpicture}
\end{center}
The parallel arrows are closed inclusions and hence proper --- the right adjoints are given by taking joins with ${\rrparenthesis 0,0 \llparenthesis}$ and ${\rrparenthesis 1,1 \llparenthesis}$ respectively. The pullback of these parallel arrows is the empty locale and so the unique map $t\colon 0 \to 1$ trivially satisfies the assumptions of \cref{prop:proper_transitive_interior_operator,rem:proper_coequaliser_and_symmetry} and hence $q$ is a proper quotient and the corresponding interior operator $q^*q_*$ is $(1)_*(0)^* \wedge (0)_*(1)^* \wedge \id$.

In terms of generators we have that $(0)^*({\rrparenthesis p,q \llparenthesis}) = 1 \iff p > 0$ and $(1)^*({\rrparenthesis p,q \llparenthesis}) = 1 \iff {q < 1}$.
Thus, we find
\[q^*q_*\colon {\rrparenthesis p,q \llparenthesis} \mapsto
    \begin{cases}
     {\rrparenthesis p,q \llparenthesis} & \text{ if $p > 0$ and $q < 1$} \\
     {\rrparenthesis p,q \llparenthesis} \wedge {\rrparenthesis 1,1 \llparenthesis} & \text{ if $p = 0$ and $q < 1$} \\
     {\rrparenthesis p,q \llparenthesis} \wedge {\rrparenthesis 0,0 \llparenthesis} & \text{ if $p > 0$ and $q = 1$} \\
     0 & \text{ if $p = 0$ and $q = 1$} \\
    \end{cases},\]
which intuitively adds $\{0\}$ to any closed interval containing $1$ and $\{1\}$ to any closed interval containing $0$.

We are now in a position to use \cref{prop:proper_quotient} to immediately obtain a presentation for $\O \T$ with generators ${\rrparenthesis p,q \llparenthesis}$ for $p,q \in \Q \cap [0,1]$ and the relations:
\begin{enumerate}
 \item ${\rrparenthesis 0, 1 \llparenthesis} = 0$,
 \item ${\rrparenthesis p,q \llparenthesis} \vee {\rrparenthesis p',q' \llparenthesis} = {\rrparenthesis p \vee p', q \wedge q' \llparenthesis}$ for $p' > 0$ and $q' < 1$,
 \item ${\rrparenthesis p,q \llparenthesis} \vee {\rrparenthesis 0,q' \llparenthesis} = {\rrparenthesis p, q \wedge q' \llparenthesis} \wedge {\rrparenthesis 1,q \llparenthesis}$,
 \item ${\rrparenthesis p,q \llparenthesis} \vee {\rrparenthesis p',1 \llparenthesis} = {\rrparenthesis p \vee p', q \llparenthesis} \wedge {\rrparenthesis p,0 \llparenthesis}$,
 \item ${\rrparenthesis p,q \llparenthesis} \wedge {\rrparenthesis p',q' \llparenthesis} = {\rrparenthesis p,q' \llparenthesis}$ for $p \le p' \le q \le q'$,
 \item ${\rrparenthesis p,q \llparenthesis} = 1$ for $p > q$,
 \item ${\rrparenthesis p,q \llparenthesis} = \dirsup[q' > q] {\rrparenthesis p,q' \llparenthesis}$ for $p, q < 1$ and ${\rrparenthesis p,q \llparenthesis} = \dirsup[p' < p] {\rrparenthesis p',q \llparenthesis}$ for $0 < p, q$.
\end{enumerate}
This can be further simplified by noting that ${\rrparenthesis 1,q \llparenthesis} = 1$ unless $q = 1$ by relation (vi) and similarly for ${\rrparenthesis p,0 \llparenthesis}$.
Thus, we can extend relation (ii) to hold for all $p,p',q,q'$ such that $(p',q),(p,q') \ne (0,1)$.
Then the remaining cases of (iii) and (iv) can be reduced to a single rule: ${\rrparenthesis 0,q \llparenthesis} \vee {\rrparenthesis p',1 \llparenthesis} = {\rrparenthesis p', q \llparenthesis} \wedge {\rrparenthesis 0,0 \llparenthesis} \wedge {\rrparenthesis 1,1 \llparenthesis}$.
\end{example}

\bibliographystyle{abbrv}
\bibliography{references}

\end{document}